\documentclass[a4paper,12pt]{article} 

\usepackage[T1]{fontenc}
\usepackage{amssymb}

\usepackage[dvips]{graphicx}

\usepackage{amsfonts, amsmath, amsthm, amssymb}
\usepackage[ruled,vlined,linesnumbered]{algorithm2e}
\usepackage{epsfig}
\usepackage[T1]{fontenc}
\usepackage[cp1250]{inputenc}

\usepackage{tikz}
\usepackage{fancyhdr}
\makeindex

\allowdisplaybreaks[1]

\newtheorem{theorem}{Theorem}

\newtheorem{corollary}[theorem]{Corollary}

\newtheorem{definition}[theorem]{Definition}

\newtheorem{example}[theorem]{Example}

\newtheorem{lemma}[theorem]{Lemma}

\newtheorem{remark}[theorem]{Remark}

\input{tcilatex}

\begin{document}

\title{Markov pairs, quasi Markov functions and inverse limits}
\author{Iztok Bani\v c, Matev\v z \v Crepnjak}
\maketitle

\begin{abstract}
We show that two inverse limits of inverse sequences of closed intervals and quasi Markov bonding functions are homeomorphic, if the inverse sequences follow the same pattern. This significantly improves Holte's result about when two inverse limits of inverse sequences with Markov interval maps as bonding functions are homeomorphic. 

Our result improves Holte's result in several directions: (1) it generalizes finite Markov partitions to Markov pairs of sets that may even be uncountable, (2) it generalizes Markov interval maps $I\rightarrow I$ to quasi Markov functions $I\rightarrow J$ (so the domains and the codomains of the bonding functions are not necessarily the same interval), (3) we no longer require the bonding functions to be surjective, and (4) we no longer require the bonding functions to be continuous.
\end{abstract}

\maketitle


\section{Introduction}

The Markov partition of a closed interval $I = [x,y]$ with respect to a function $f:I\rightarrow I$ is usually given by the points $x = x_0 < x_1 < \ldots < x_n = y$ in $I$ such that all the restrictions $f|_{[x_{i-1},x_i]}$ of $f$ to $[x_{i-1},x_i]$ are homeomorphisms from $[x_{i-1},x_i]$ onto some interval $[x_k, x_{\ell}]$. Since a Markov partition is usually given by a finite collection of points $A = \{x_0, x_1, \ldots , x_n\} \subseteq I$ we usually refer to $A$ as a Markov partition. If a continuous function $f$ has a Markov partition $A$, then we say that $f$ is a Markov interval map with respect to $A$. 

For a dynamical system $(I, f)$, where $I$ is a closed interval and $f: I \rightarrow I$ a continuous function, a Markov partition of $I$ with respect to $f$ (if it exists) is a well-known tool in the dynamical system theory that allows the methods of symbolical dynamics to be used to study the dynamical system $(I,f)$. 
For more information about Markov partitions in dynamical systems and symbolical dynamics, see \cite{henk,lind}.

In \cite{Holte}, Holte introduced when two given Markov interval maps follow the same pattern ($f$ with respect to $A=\{a_0 , a_1 , \ldots , a_m \}$ and $g$ with respect to $B=\{b_0 , b_1 , \ldots , b_m\}$ follow the same pattern (with respect to $A$ and $B$), if $f(a_j)=a_k$ if and only if $g(b_j)=b_k$ for all $j$ and $k$) and proved the following:
 \begin{theorem}\label{holte}
Let $I$ and $J$ be closed intervals. If
\begin{enumerate}
\item $ \{f_n\}_{n=1}^{\infty}$ is a sequence of surjective maps $I\rightarrow I$ that are all Markov interval maps with respect to $A\subseteq I$, which is a Markov partition for each $f_n$,
\item $ \{g_n\}_{n=1}^{\infty}$ is a sequence of surjective maps $J\rightarrow J$ that are all Markov interval maps with respect to $B\subseteq J$, which is a Markov partition for each $g_n$, and
\item for each $n$, $f_n$ and $g_n$ are Markov interval maps that follow the same pattern,
\end{enumerate}
then $\varprojlim \{I,f_n\}_{n=1}^{\infty}$ is homeomorphic to $\varprojlim \{J,g_n\}_{n=1}^{\infty}$.
\end{theorem}

As a main result of this paper, we prove Theorem \ref{main} which generalizes Theorem \ref{holte}. 
Holte's proof of Theorem \ref{holte} requires all the bonding functions to be surjective and continuous.
Our new method to prove Theorem \ref{main} shows (among other things) that continuity as well as surjectivity 
of bonding functions are not really required.

Note also, that if $\{I_n,f_n\}_{n=1}^{\infty}$ is an inverse sequence of intervals with Markov interval maps, then $I_k = I_{\ell}$ for any positive integers $k$ and $\ell$. Therefore the domains and codomains of all bonding functions $f_n$ have to be equal.   
In this paper we eliminate this restriction by generalizing the notion of Markov partitions $A\subseteq I$ for a function $I\rightarrow I$ to the notion of Markov pairs $(A,B)\subseteq I\times J$ for a function $I\rightarrow J$ in such a way that any Markov partition $A$ will produce a Markov pair $(A,A)$.
With this notion we introduce quasi Markov functions which are a generalization of Markov interval maps.
 
Hence, we generalize Holte's result (Theorem \ref{holte}) in the following directions:
\begin{enumerate}
\item by generalizing finite Markov partitions $A\subseteq I$ to Markov pairs $(A,B)$, where $A$ and $B$ are any totally disconnected closed subsets of $I$ and $J$, respectively,
\item by generalizing Markov interval maps $I\rightarrow I$ to quasi Markov maps $I\rightarrow J$ (so the domains and the codomains of the bonding functions are not necessarily the same interval), 
\item by omitting the condition that all the bonding functions are surjective, and 
\item by omitting the condition that all the bonding functions are continuous.
\end{enumerate}

As a main tool to prove our main result (Theorem \ref{main}), we use functions $I \rightarrow 2^{J}$ instead of continuous functions $I \rightarrow J$. Hence, our result generalizes the Holte's result also in another direction by
\begin{enumerate}
\item[(5)] generalizing Markov interval maps $I\rightarrow I$ to quasi Markov functions $I\rightarrow 2^J$.
\end{enumerate}

Some generalizations of the Holte's result have already been introduced, for examples see \cite{BanicLunder} (where Markov interval maps $I\rightarrow I$ were replaced by so called generalized Markov functions $I\rightarrow 2^I$) and \cite{CrepnjakLunder} (where finite Markov partitions $A\subseteq I$  were replaced by countable closed sets $A$ in $I$ which only have finitely many limit points; then Markov interval maps $I\rightarrow I$ were replaced by so called countably Markov interval functions $I\rightarrow 2^I$).

\section{Definitions and notation}

	In present paper we always deal with {\it nondegenerate closed intervals} $[x,y]\subseteq \mathbb R$, i.e.\ $x<y$.
	A subspace $A$ of a closed interval is {\it totally disconnected}, if it is not connected and each component of $A$ is a singleton.
	
	A function $f:I\rightarrow J$ from an interval to an interval is {\it strictly increasing} ({\it strictly decreasing}), if for any $s,t\in I$, from $s<t$ ($s>t$) it follows $f(s)<f(t)$ ($f(s)>f(t)$). We say that $f$ is {\it strictly monotone}, if it is either strictly increasing or strictly decreasing.
		
	We always denote the {\it left-hand limit of $f:I\rightarrow J$ at a point $a \in I$} by 
	$\displaystyle \lim_{t \to a-} f(t)$ and the {\it right-hand limit of $f$ at a point $a \in I$} by $\displaystyle \lim_{t \to a+} f(t)$. 
	
	In present paper we deal with {\it inverse sequences} of closed intervals and functions, i.e.\ double sequences $\{I_n,f_n\}_{n=1}^{\infty}$ of closed intervals $I_n$ 
	and (not necessarily continuous) functions $f_n:I_{n+1}\rightarrow {I_n}$. The {\it inverse limit} of an inverse sequence $\{I_n, f_n\}_{n=1} ^\infty$ is defined to be the 
	subspace of $\prod_{n=1}^\infty I_n$ of all points $\mathbf x=(x_1, x_2, x_3, \ldots ) \in \prod_{n=1}^\infty I_n$, 
	such that $x_n = f_n(x_{n+1})$ for each $n$. The inverse limit is denoted by 
	$\varprojlim\{I_n, f_n\}_{n=1}^{\infty}$. Note that since the bonding functions may not be continuous, it may happen that $\varprojlim\{I_n, f_n\}_{n=1}^{\infty}$ is empty
	(such an inverse limit is presented in Example \ref{ex:3}). 
	
	If $I$ is a closed interval then $2^{I}$ denotes the family of all nonempty closed subsets of $I$.
	
	Our results also include functions $F:I \rightarrow 2^{J}$, where $I$ and $J$ are two closed intervals.
	The {\it graph} of such a function $F$ is defined to be the subset of $I \times J$ defined by 
	$$
	\Gamma(F) = \{(x,y) \ | \ y \in F(x), \ x \in I \}.
	$$
	Note that the graph of $F$ is not defined in the usual sense as the subset $\{(x,Y) \ | \ Y = F(x), \ x \in I \}$ of $I \times 2^{J}$.
	Such a function $F$ is {\em upper semicontinuous} 
	if the graph of $F$	is a closed subset of $I \times J$, see \cite[Theorem 1.2]{ingram}.

If $F:I\rightarrow 2^J$ is a function, where for each $t \in I$, the image $F(t)$ is a one-point subset of $J$, 
	then we can identify it with the function $f:I\rightarrow J$, where $F(t)=\{f(t)\}$ for any $t\in I$. 
	Conversely, any function $f:I\rightarrow J$ can be identified with the function $F:I\rightarrow 2^J$, defined by $F(t)=\{f(t)\}$ for any $t \in I$ (note that the graphs of $f$ and $F$ are the same). 
	Obviously, in this situation, the following are equivalent (since $\Gamma(f)=\Gamma(F)$ is a closed subset of $I\times J$):
	\begin{enumerate}
	\item $F$ is upper semicontinuous,
	\item $f$ is continuous.
	\end{enumerate}
In this case we denote the {\it left-hand limit of $f:I\rightarrow J$ at a point $a \in I$} also by 
	$\displaystyle \lim_{t \to a-} F(t)$ and the {\it right-hand limit of $f$ at a point $a \in I$} also by $\displaystyle \lim_{t \to a+} F(t)$. 		
		
	In our results, we mostly use {\it generalized inverse sequences}, i.e.\ double sequences $\{I_n,F_n\}_{n=1}^{\infty}$, 
	where $I_n$ is a closed interval and $F_n:I_{n+1}\rightarrow 2^{I_n}$ is a (not necessarily upper semicontinuous) function for each $n$. 
	The  {\it inverse limit} of such a generalized inverse sequence $\{I_n, F_n\}_{n=1} ^\infty$ is defined to be the 
	subspace of $\prod_{n=1}^\infty I_n$ of all $\mathbf x=(x_1, x_2, x_3, \ldots ) \in \prod_{n=1}^\infty I_n$, 
	such that $x_n \in F_n(x_{n+1})$ for each $n$. Also in this case, the inverse limit is denoted by 
	$\varprojlim\{I_n, F_n\}_{n=1}^{\infty}$. Note that also in this situation, since the bonding functions may not be upper semicontinuous, it may happen that $\varprojlim\{I_n, F_n\}_{n=1}^{\infty}$ is empty. 
			
	Inverse limits with upper semicontinuous bonding functions were first introduced in 2004 by Mahavier and later by Ingram and Mahavier. 
	Since their introduction many authors have been interested in this area and many papers appeared (for more details and other references see \cite{ingram}).

\section{Markov pairs and quasi Markov functions}

	In this section we introduce the concepts of Markov pairs and quasi Markov functions, prove some of their properties and give several examples. First we introduce the concept of Markov pairs which generalizes the well-known concept of Markov partitions.

	\begin{definition}\label{d:pair}
		Let $I_1=[x_1,y_1]$ and $I_2=[x_2,y_2]$ be closed intervals, $F:I_1\rightarrow 2^{I_2}$ a (not necessarily upper semicontinuous) function, 
		$A_1$ a totally disconnected closed subset of $I_1$ such that $x_1,y_1\in A_1$, and $A_2$ a totally disconnected closed subset of $I_2$
		such that $x_2,y_2\in A_2$. We say that {\em $(A_1,A_2)$ is a Markov pair for $F$\/}, if 
		\begin{enumerate}
			\item\label{iki} for each $a\in A_1$, $F(a)$ is a is a closed subset of $J_2$ 
				such that for each $ t\in F(a)$, if $t\in \textup{Bd}(F(a))$, then $t\in A_2$, 
				\item for each component $C$ of $I_1\setminus A_1$, the restriction $F|_C : C \rightarrow 2^{I_2}$ of $F$ to $C$ is a strictly monotone continuous function
			(meaning that $F|_C(t)$ is a singleton $\{s_t\}$ for each $t \in C$ and that $f:C\rightarrow I_2$, defined by $f(t)=s_t$ for each $t\in C$, is a strictly monotone continuous function),
			\item for each component $C=(a,b)$ of $I_1\setminus A_1$, $a,b\in A_1$,  $\displaystyle \lim_{t\to a+}F(t),\lim_{t\to b-}F(t)\in A_2$.
		\end{enumerate} 
	\end{definition}
	As seen before, functions $f:I\rightarrow J$ can be interpreted as functions $F:I\rightarrow 2^J$, $F(t)=\{f(t)\}$ for each $t\in I$. Therefore, Definition \ref{d:pair} also includes such functions. Note that in this case, (\ref{iki}) from Definition \ref{d:pair} is equivalent to $f(a)\in A_2$ for each $a\in A_1$. 
	
	\begin{example}
		Let $f: I \rightarrow I$ be a Markov interval map with respect to a Markov partition $A$. 
		Then, since $A$ is a finite subset of $I$ with at least two elements, it is a totally disconnected closed subset of $I$ for which obviously all requirements from Definition \ref{d:pair} are satisfied. Therefore $(A,A)$ is a Markov pair for $f$. 
	\end{example}
	
	\begin{definition}\label{d:markov}
		Let $I_1$ and $I_2$ be closed intervals, and $F:I_1\rightarrow 2^{I_2}$ a (not necessarily upper semicontinuous) function. 
		We say that $F$ is {\em quasi Markov\/}, if there is a Markov pair for $F$.
	
		We say that $F$ is {\em quasi Markov with respect to $(A_1,A_2)$\/}, if $F$ is quasi Markov and $(A_1,A_2)$ is a Markov pair for $F$.
	\end{definition}
	Again, since functions $I\rightarrow J$ can be interpreted as functions $I\rightarrow 2^J$,  Definition \ref{d:markov} also includes such functions. 
	\begin{example}
		Let $f: I \rightarrow I$ be a Markov interval map with respect to a Markov partition $A$. 
		Then $f$ is a quasi Markov function with respect to $(A,A)$.
	\end{example}
	Next we introduce when two generalized inverse sequences of quasi Markov functions follow the same pattern. 
	\begin{definition}\label{d:pattern}
		Let $\{I_n,F_n\}_{n=1}^{\infty}$ and $\{J_n,G_n\}_{n=1}^{\infty}$ be generalized inverse sequences of closed intervals $I_n$ and $J_n$, and quasi Markov functions $F_n:I_{n+1}\rightarrow 2^{I_n}$ and $G_n:J_{n+1}\rightarrow 2^{J_n}$.	
		We say that $\{I_n,F_n\}_{n=1}^{\infty}$ and $\{J_n,G_n\}_{n=1}^{\infty}$	{\em follow the same pattern with respect to
		$(A_n)_{n=1}^{\infty}\in \prod_{n=1}^{\infty}2^{I_n}$ and $(B_n)_{n=1}^{\infty}\in \prod_{n=1}^{\infty}2^{J_n}$}, if 
		\begin{enumerate}
		\item\label{100} for each $n$, 		
		$(A_{n+1},A_n)$ is a Markov pair for $F_n$ and $(B_{n+1},B_n)$ is a Markov pair for $G_n$, 
		\item\label{200} there is a strictly increasing bijection $\tau_1:A_1 \rightarrow B_1$, and 
		\item\label{300} for each $n$, 
		there are strictly increasing bijections $\varphi_n:A_{n+1}\rightarrow A_n$ and $\psi_n : B_{n+1}\rightarrow B_n$ 
		such that	
		\begin{enumerate}
		\item\label{a} for each $a\in A_{n+1}$, there is a homeomorphism $h: F_n(a)\rightarrow G_n(\tau_{n+1}(a))$ such that
			$$h(\textup{Bd}(F_n(a))) = \tau_n(\textup{Bd}(F_n(a))) = \textup{Bd}(G_n(\tau_{n+1}(s))),$$

			\item\label{c} for each component $C=(a,b)$ of $I_{n+1}\setminus A_{n+1}$, $a,b\in A_{n+1}$, 			
			$$ 
			c=\lim_{t\to a+}F_n(t) \textup{ if and only if } \tau_n(c)=\lim_{t\to \tau_{n+1}(a)+}G_n(t),
			$$
			and
			$$
			c=\lim_{t\to b-}F_n(t) \textup{ if and only if } \tau_n(c)=\lim_{t\to \tau_{n+1}(b)-}G_n(t),
			$$
		\end{enumerate}
		where 
		$
		\tau_{n+1} = \psi_{n}^{-1} \circ\psi_{n-1}^{-1} \circ \ldots \circ \psi_{1}^{-1} \circ\tau_1 \circ\varphi_{1} \circ \varphi_{2}\circ  \ldots \circ \varphi_n : A_{n+1} \rightarrow B_{n+1}
		$
		for each positive integer $n$, see (\ref{equat}).
			\end{enumerate}

		\begin{center}				
		\begin{equation}	\label{equat}			
		\begin{tikzpicture}[node distance=5.0em, auto]
  \node (X) {$A_1$};
  \node (Y) [right of=X] {$A_{2}$};
  \node (Z) [right of=Y] {$A_{3}$};
	 \node (a) [right of=Z] {$\ldots$};
	 \node (b) [right of=a] {$A_{n}$};
	 \node (c) [right of=b] {$A_{n+1}$};
	 \node (d) [right of=c] {$\ldots$};
  \draw[<-] (X) to node {$\varphi_1$} (Y);
  \draw[<-] (Y) to node {$\varphi_2$} (Z);
  \draw[<-] (Z) to node {$\varphi_3$} (a);
	\draw[<-] (a) to node {$\varphi_{n-1}$} (b);
	\draw[<-] (b) to node {$\varphi_n$} (c);
	\draw[<-] (c) to node {$\varphi_{n+1}$} (d);
	\node (X1) [below of=X] {$B_1$};
  \node (Y1) [right of=X1] {$B_{2}$};
  \node (Z1) [right of=Y1] {$B_{3}$};
	 \node (a1) [right of=Z1] {$\ldots$};
	 \node (b1) [right of=a1] {$B_{n}$};
	 \node (c1) [right of=b1] {$B_{n+1}$};
	 \node (d1) [right of=c1] {$\ldots$};
  \draw[<-] (X1) to node {$\psi_1$} (Y1);
  \draw[<-] (Y1) to node {$\psi_2$} (Z1);
  \draw[<-] (Z1) to node {$\psi_3$} (a1);
	\draw[<-] (a1) to node {$\psi_{n-1}$} (b1);
	\draw[<-] (b1) to node {$\psi_n$} (c1);
	\draw[<-] (c1) to node {$\psi_{n+1}$} (d1);
	\draw[->] (X) to node {$\tau_1$} (X1);
  \draw[->] (Y) to node {$\tau_2$} (Y1);
  \draw[->] (Z) to node {$\tau_3$} (Z1);
	\draw[->] (b) to node {$\tau_n$} (b1);
	\draw[->] (c) to node {$\tau_{n+1}$} (c1);
\end{tikzpicture}
	\end{equation}
	\end{center}

	We say that $\{I_n,F_n\}_{n=1}^{\infty}$ and $\{J_n,G_n\}_{n=1}^{\infty}$	{\em follow the same pattern\/}, if $\{I_n,F_n\}_{n=1}^{\infty}$ and $\{J_n,G_n\}_{n=1}^{\infty}$	follow the same pattern with respect to some
		$(A_n)_{n=1}^{\infty}\in \prod_{n=1}^{\infty}2^{I_n}$ and $(B_n)_{n=1}^{\infty}\in \prod_{n=1}^{\infty}2^{J_n}$.
	\end{definition}
	
	\begin{remark}\label{b}
		Note that from (\ref{a}) in Definition \ref{d:pattern} it follows that for each $a \in A_{n+1}$ and for each strictly increasing homeomorphism $h: I_n \rightarrow J_n$
		such that $h(\textup{Bd}(F_n(a))) = \textup{Bd}(G_n(\tau_{n+1}(a)))$, the restriction $h|_{F_n(a)} : F_n(a) \rightarrow G_n(\tau_{n+1}(a))$ of $h$ is also a homeomorphism.
	\end{remark}
	
	Also in this case, Definition \ref{d:pattern} includes inverse sequences $\{I_n,f_n\}_{n=1}^{\infty}$ and $\{J_n,g_n\}_{n=1}^{\infty}$ of closed intervals and quasi Markov functions $f_n:I_{n+1}\rightarrow I_n$ and $g_n:J_{n+1}\rightarrow J_n$. Note that in this case, (\ref{a}) from Definition \ref{d:pattern} is equivalent to the requirement that for each $a\in A_{n+1}$, 
			
			$$
			f_n(a)=t \textup{ if and only if } g_n(\tau_{n+1}(a))=\tau_n(t),
			$$

		The following theorem easily follows.
	\begin{theorem}\label{corcor}
Suppose that $\{I,f_n\}_{n=1}^{\infty}$ and $\{J,g_n\}_{n=1}^{\infty}$ are two inverse sequences of closed intervals and quasi Markov functions $f_n:I\rightarrow I$ and $g_n:J\rightarrow J$. If
		\begin{enumerate}
			\item\label{102} $(A,A)$ is a Markov pair for $f_n$ and $(B,B)$ is a Markov pair for $g_n$ for each positive integer $n$,
			\item\label{202} there is a strictly increasing bijection $\tau : A\rightarrow B$ such that		 
		\begin{enumerate}
			\item\label{a101} for each $a\in A$, 
			
			$$
			f_n(a)=t \textup{ if and only if } g_n(\tau(a))=\tau(t),
			$$
			
			\item\label{b101} for each component $C=(a,b)$ of $I\setminus A$, $a,b\in A$, 			
			$$ 
			c=\lim_{t\to a+}f_n(t) \textup{ if and only if } \tau(c)=\lim_{t\to \tau(a)+}g_n(t),
			$$
			and
			$$
			c=\lim_{t\to b-}f_n(t) \textup{ if and only if } \tau(c)=\lim_{t\to \tau(b)-}g_n(t),
			$$
		\end{enumerate}
			\end{enumerate}
			then $\{I,f_n\}_{n=1}^{\infty}$ and $\{J,g_n\}_{n=1}^{\infty}$ follow the same pattern with respect to $(A)_{n=1}^{\infty}\in \prod_{n=1}^{\infty}2^{I}$ and $(B)_{n=1}^{\infty}\in \prod_{n=1}^{\infty}2^{J}$.
		\end{theorem}
		\begin{proof}
Obviously (\ref{100}) from Definition \ref{d:pattern} holds true. If $\tau_1=\tau$ and for each positive integer $n$, $\varphi_n=\textup{id}_A$ and $\psi_n=\textup{id}_B$, then (\ref{200}) and (\ref{300}) from Definition \ref{d:pattern} are also satisfied. 
		\end{proof}

	\begin{example}
		Let $f: I \rightarrow I$ and $g:J \rightarrow J$ be Markov interval maps with respect to Markov partitions $A$ for $f$ and $B$ for $g$ that follow the same pattern 
		with respect to $A$ and $B$. Then obviously, (\ref{102}) from Theorem \ref{corcor} is satisfied and there is a unique strictly increasing bijection $\tau : A \rightarrow B$ (since $|A| = |B| < \infty$). Then also (\ref{202}) from Theorem \ref{corcor} holds true (since $f$ and $g$ follow the same pattern with respect to $A$ and $B$).
		
	So, all the conditions from Theorem \ref{corcor} are satisfied. Therefore, the inverse sequences
		$\{I,f\}_{n=1}^{\infty}$ and $\{J,g\}_{n=1}^{\infty}$	also follow the same pattern with respect to $(A)_{n=1}^{\infty}\in \prod_{n=1}^{\infty}2^{I}$ and $(B)_{n=1}^{\infty}\in \prod_{n=1}^{\infty}2^{J}$.
	\end{example}
	
\section{The main result}
	In this section we prove Theorem \ref{main} -- the main result of the paper. We will need the following lemma (which generalizes \cite[Theorem 4.5.]{{ingram}}, where generalized inverse sequences with upper semicontinuous functions are used) in its proof.
	\begin{lemma}\label{tool}
	Let $\{I_n,F_n\}_{n=1}^{\infty}$ and $\{J_n,G_n\}_{n=1}^{\infty}$ be two generalized inverse sequences (with bonding functions that are not necessarily upper semicontinuos). If for each positive integer $n$, there is a homeomorphism $h_n:I_n\rightarrow J_n$ such that $h_n \circ F_n = G_n \circ h_{n+1}$, then $\varprojlim\{I_n,F_n\}_{n=1}^{\infty}$ and $\varprojlim\{J_n,G_n\}_{n=1}^{\infty}$ are homeomorphic.
	\end{lemma}
	\begin{proof}
  For any $\mathbf x=(x_1,x_2,x_3,\ldots)\in \varprojlim\{I_n,F_n\}_{n=1}^{\infty}$ (if such an $\mathbf x$ exists) we define $$
	h(\mathbf x)=(h_1(x_1),h_2(x_2),h_3(x_3),\ldots).
	$$
	Obviously, $h:\varprojlim\{I_n,F_n\}_{n=1}^{\infty}\rightarrow \varprojlim\{J_n,G_n\}_{n=1}^{\infty}$, since for any positive integer $n$, $h_n \circ F_n = G_n \circ h_{n+1}$ and therefore $h_n(x_n)\in G_n(h_{n+1}(x_{n+1}))$. It follows that
	\begin{enumerate}
	\item if $\varprojlim\{I_n,F_n\}_{n=1}^{\infty}\neq \emptyset$, then $\varprojlim\{J_n,G_n\}_{n=1}^{\infty}\neq \emptyset$, and 
	\item $h:\varprojlim\{I_n,F_n\}_{n=1}^{\infty}\rightarrow \varprojlim\{J_n,G_n\}_{n=1}^{\infty}$ is a continuous function (if the inverse limit $\varprojlim\{I_n,F_n\}_{n=1}^{\infty}\neq \emptyset$).
	\end{enumerate}
	Next, for any $\mathbf y=(y_1,y_2,y_3,\ldots)\in \varprojlim\{J_n,G_n\}_{n=1}^{\infty}$ (if such an $\mathbf y$ exists) we define $$
	g(\mathbf y)=(h_1^{-1}(y_1),h_2^{-1}(y_2),h_3^{-1}(y_3),\ldots).
	$$
	We show that $g:\varprojlim\{J_n,G_n\}_{n=1}^{\infty}\rightarrow \varprojlim\{I_n,F_n\}_{n=1}^{\infty}$. Since for any positive integer $n$, $h_n \circ F_n = G_n \circ h_{n+1}$ holds true, it follows that $h_n^{-1}\circ G_n=F_n\circ h_{n+1}^{-1}$. Therefore $h_n^{-1}(y_n)\in F_n(h_{n+1}^{-1}(y_{n+1}))$. This means that
	\begin{enumerate}
	\item if $\varprojlim\{J_n,G_n\}_{n=1}^{\infty}\neq \emptyset$, then $\varprojlim\{I_n,F_n\}_{n=1}^{\infty}\neq \emptyset$, and 
	\item $g:\varprojlim\{J_n,G_n\}_{n=1}^{\infty}\rightarrow \varprojlim\{I_n,F_n\}_{n=1}^{\infty}$ is a continuous function (if the inverse limit $\varprojlim\{J_n,G_n\}_{n=1}^{\infty}\neq \emptyset$).
	\end{enumerate}
	So, if one of the inverse limits is empty, so is the other one and they are homeomorphic. If one of them is non-empty, then so is the other one. In this case,
		$$
		g(h(\mathbf x))=(h_1^{-1}(h_1(x_1)),h_2^{-1}(h_2(x_2)),h_3^{-1}(h_3(x_3)),\ldots)=\mathbf x
		$$
		for each $\mathbf x\in \varprojlim\{I_n,F_n\}_{n=1}^{\infty}$ and 
		$$
		h(g(\mathbf y))=(h_1(h_1^{-1}(y_1)),h_2(h_2^{-1}(y_2)),h_3(h_3^{-1}(y_3)),\ldots)=\mathbf y
		$$
		for each $\mathbf y\in \varprojlim\{J_n,G_n\}_{n=1}^{\infty}$. Therefore $h$ is a homeomorphism.
	\end{proof}
	\begin{theorem}\label{main}
		Let $\{I_n,F_n\}_{n=1}^{\infty}$ and $\{J_n,G_n\}_{n=1}^{\infty}$ be two generalized inverse sequences of closed intervals and quasi Markov functions
		that follow the same pattern. Then the inverse limits $\varprojlim\{I_n,F_n\}_{n=1}^{\infty}$ and $\varprojlim\{J_n,G_n\}_{n=1}^{\infty}$ are homeomorphic. 
	\end{theorem}
	
	\begin{proof}
		Without loss of generality suppose that both inverse limits are non-empty. Since the inverse sequences $\{I_n,F_n\}_{n=1}^{\infty}$ and $\{J_n,G_n\}_{n=1}^{\infty}$ follow the same pattern,
		there are $(A_n)_{n=1}^{\infty}\in \prod_{n=1}^{\infty}2^{I_n}$ and $(B_n)_{n=1}^{\infty}\in \prod_{n=1}^{\infty}2^{J_n}$ 
		satisfying  Definition \ref{d:pattern}.
		
				Let $\tau_1 : A_1 \rightarrow B_1$ be a strictly increasing bijection and for each positive integer $n$, let $\varphi_n:A_{n+1}\rightarrow A_n$ and $\psi_n : B_{n+1}\rightarrow B_n$ be strictly increasing bijections such that (\ref{a})--(\ref{c}) from Definition \ref{d:pattern} are satisfied.

		Next, let $h_1 : I_1 \to J_1$ be any homeomorphism, such that for each $a \in A_1$, $h_1(a) = \tau_1 (a)$.
		Obviously such a homeomorphism exists: For any component $C$ of $I_1 \setminus A_1$, it follows that $\textup{Cl}(C) = [a_1,a_2]$ for some $a_1,a_2 \in A_1$. 
		Then we can define $h_1$ on $\textup{Cl}(C)$ to be the linear function taking $a_1$ to $\tau_1(a_1)$ and $a_2$ to $\tau_1(a_2)$.
		Since $h_1$ is now defined in such a way that it is a strictly increasing surjection $I_1 \rightarrow J_1$, it is therefore a homeomorphism.

		Suppose that we have already defined homeomorphisms $h_i:I_i\rightarrow J_i$, $i=1,2,3,\ldots,n$, such that for each $a \in A_i$, $h_i(a) = \tau_i(a)$ for any $i=1,2,3,\ldots ,n$, and $h_i \circ F_i = G_i \circ h_{i+1}$ for any $i=1,2,3,\ldots ,n-1$. Then we construct the homeomorphism $h_{n+1}$ as follows.
	
		Obviously, $\tau_{n+1}:A_{n+1}\rightarrow B_{n+1}$ is also a strictly increasing bijection (it is a composition of strictly increasing bijections). 
			Next, let $h_{n+1} : I_{n+1} \to J_{n+1}$ be the function defined by
		$$h_{n+1} (t) =
			\begin{cases}
				\tau_{n+1} (t)\text{;} & t \in A_{n+1}, \\
				(G_n|_{(\tau_{n+1}(a_1),\tau_{n+1}(a_2))})^{-1} (h_n(F_n(t))) \text{;} & t \in C,
			\end{cases}
		$$
		where $C=(a_1,a_2)$ is the component of $I_{n+1} \backslash A_{n+1}$, $a_1,a_2\in A_{n+1}$, such that $t \in C$.
		
		First we prove that $h_{n+1}$ is surjective. Let $t_0\in J_{n+1}$. If $t_0\in B_{n+1}$, then $s_0=\tau_{n+1}^{-1}(t_0)\in I_{n+1}$ is the point such that $t_0=\tau_{n+1}(s_0)=h_{n+1}(s_0)$. If $t_0\not \in B_{n+1}$, then there is a component $D=(b_1,b_2)$ of $J_{n+1}\setminus B_{n+1}$, $b_1,b_2\in B_{n+1}$, such that $t_0\in D$. Let $a_1=\tau_{n+1}^{-1}(b_1)$ and  $a_2=\tau_{n+1}^{-1}(b_2)$. Then $C=(a_1,a_2)$ is a component of $I_{n+1}\setminus A_{n+1}$ (since $\tau_{n+1}:A_{n+1}\rightarrow B_{n+1}$ is a strictly increasing bijection). Since $F_{n}$ and $G_n$ are quasi Markov with respect to $(A_{n+1},A_n)$ and $(B_{n+1},B_n)$, respectively, it follows that the restrictions $F_{n}|_{C}$ and $G_n|_{D}$ are continuous and strictly monotone - either both strictly increasing or both strictly decreasing continuous functions (taking into account (\ref{c}) from Definition \ref{d:pattern} and the fact that $F_n$ and $G_n$ are quasi Markov). Then $h:D\rightarrow C$, defined by $h(t)=(F_{n}|_{C})^{-1}(h_n^{-1}(G_n|_{D}(t)))$ for each $t\in D$, is a strictly increasing function. Taking $s_0=h(t_0)$, we get that 
		$$
		h_{n+1}(s_0)=((G_n|_{D})^{-1} \circ h_n\circ F_n\circ (F_{n}|_{C})^{-1}\circ h_n^{-1}\circ G_n|_{D})(t_0)=t_0
		$$
We have proved that $h_{n+1}$ is surjective. Also, since $h_{n+1}$ is strictly increasing on each component $C$ of $I_{n+1}\setminus A_{n+1}$, it is a strictly increasing function $I_{n+1}\rightarrow J_{n+1}$.
	Therefore, $h_{n+1}$ is defined in such a way that it is a strictly increasing surjection $I_{n+1} \rightarrow J_{n+1}$ and it is therefore a homeomorphism.

	Next we show that $h_n \circ F_n = G_n \circ h_{n+1}$ for each positive integer $n$, i.e.\ we prove that the diagram (\ref{equat1}) commutes.

		\begin{center}
		\begin{equation}	\label{equat1}	
		\begin{tikzpicture}[node distance=5.00em, auto]
  \node (X) {$I_1$};
  \node (Y) [right of=X] {$I_{2}$};
  \node (Z) [right of=Y] {$I_{3}$};
	 \node (a) [right of=Z] {$\ldots$};
	 \node (b) [right of=a] {$I_{n}$};
	 \node (c) [right of=b] {$I_{n+1}$};
	 \node (d) [right of=c] {$\ldots$};
  \draw[<-] (X) to node {$F_1$} (Y);
  \draw[<-] (Y) to node {$F_2$} (Z);
  \draw[<-] (Z) to node {$F_3$} (a);
	\draw[<-] (a) to node {$F_{n-1}$} (b);
	\draw[<-] (b) to node {$F_n$} (c);
	\draw[<-] (c) to node {$F_{n+1}$} (d);
	\node (X1) [below of=X] {$J_1$};
  \node (Y1) [right of=X1] {$J_{2}$};
  \node (Z1) [right of=Y1] {$J_{3}$};
	 \node (a1) [right of=Z1] {$\ldots$};
	 \node (b1) [right of=a1] {$J_{n}$};
	 \node (c1) [right of=b1] {$J_{n+1}$};
	 \node (d1) [right of=c1] {$\ldots$};
  \draw[<-] (X1) to node {$G_1$} (Y1);
  \draw[<-] (Y1) to node {$G_2$} (Z1);
  \draw[<-] (Z1) to node {$G_3$} (a1);
	\draw[<-] (a1) to node {$G_{n-1}$} (b1);
	\draw[<-] (b1) to node {$G_n$} (c1);
	\draw[<-] (c1) to node {$G_{n+1}$} (d1);
	\draw[->] (X) to node {$h_1$} (X1);
  \draw[->] (Y) to node {$h_2$} (Y1);
  \draw[->] (Z) to node {$h_3$} (Z1);
	\draw[->] (b) to node {$h_n$} (b1);
	\draw[->] (c) to node {$h_{n+1}$} (c1);
\end{tikzpicture}
	\end{equation}
	\end{center}

Let $n$ be a positive integer and let $t\in I_{n+1}$ be any point. If $t\in A_{n+1}$, then $(h_n \circ F_n)(t)=h_n(F_n(t))$, and $(G_n \circ h_{n+1})(t)=G_n(h_{n+1}(t))=G_n(\tau_{n+1}(t))$. By (\ref{a}) of Definition \ref{d:pattern}, $G_n(\tau_{n+1}(t))$ is homeomorphic to $F_n(t)$. 
Even more, $h_n$ is a strictly increasing homeomorphism such that  
$$h_n(\textup{Bd}(F_n(t))) = \tau_n(\textup{Bd}(F_n(t))) = \textup{Bd}(G_n(\tau_{n+1}(t))).$$ 
Therefore, by Remark \ref{b}, $h_{n}:I_{n}\rightarrow J_{n}$ is a homeomorphism taking $F_n(t)$ onto $G_n(\tau_{n+1}(t))$. So, $(h_n \circ F_n)(t)=(G_n \circ h_{n+1})(t)$. If $t\not \in A_{n+1}$, then let $C$ be the component of $I_{n+1}\setminus A_{n+1}$ such that $t\in C$. Then $(G_n \circ h_{n+1})(t)=G_n(h_{n+1}(t))=G_n((G_n|_{(\tau_{n+1}(a_1),\tau_{n+1}(a_2))})^{-1} (h_n(F_n(t))))=h_n(F_n(t))$. We have proved that	$h_n \circ F_n = G_n \circ h_{n+1}$.

		By Lemma \ref{tool}, the inverse limits $\varprojlim\{I_n,F_n\}_{n=1}^{\infty}$ and $\varprojlim\{J_n,G_n\}_{n=1}^{\infty}$ are homeomorphic. 
	\end{proof}

The following corollary easily follows.

	\begin{corollary}\label{cor}
		Let $\{I_n,f_n\}_{n=1}^{\infty}$ and $\{J_n,g_n\}_{n=1}^{\infty}$ be two inverse sequences of closed intervals and quasi Markov functions $f_n:I_{n+1}\rightarrow I_n$ and $g_n:J_{n+1}\rightarrow J_n$ 
		that follow the same pattern. Then $\varprojlim\{I_n,f_n\}_{n=1}^{\infty}$ and $\varprojlim\{I_n,g_n\}_{n=1}^{\infty}$ are homeomorphic. 
	\end{corollary}
\begin{proof}
The functions $f_n:I_{n+1}\rightarrow I_n$ and $g_n:J_{n+1}\rightarrow J_n$ may be interpreted as functions $F_n:I_{n+1}\rightarrow 2^{I_n}$ and $G_n:J_{n+1}\rightarrow 2^{J_n}$ ($F_n(t)=\{f_n(t)\}$ for each $t\in I_{n+1}$ and $G_n(t)=\{g_n(t)\}$ for each $t\in J_{n+1}$). Therefore, by Theorem \ref{main}, $\varprojlim\{I_n,f_n\}_{n=1}^{\infty}$ and $\varprojlim\{I_n,g_n\}_{n=1}^{\infty}$ are homeomorphic.
\end{proof}
We conclude the paper with the following three illustrative examples. In Example \ref{ex:1}, continuous functions that are not surjective are used as bonding functions, and in Example \ref{ex:2}, functions that are neither surjective nor continuous are used as bonding functions. In both examples, all the inverse limits are non-empty. In Example \ref{ex:3}, the quasi Markov function $f:[0,1]\rightarrow [0,1]$ is constructed such that $\varprojlim\{[0,1],f\}_{n=1}^{\infty}$ is empty.
\begin{example}\label{ex:1}
Let $f,g:[0,1]\rightarrow [0,1]$ be skew-tent functions with top vertices $(a,b)$ and $(c,d)$, respectively, where $a>b>0$, $c>d>0$, $0<1-a<b$ and $0<1-c<d$ (i.e.\ the graph of $f$ is the union of two straight line segments, one from $(0,0)$ to $(a,b)$ and the other one from $(a,b)$ to $(1,0)$, and the graph of $g$ is the union of two straight line segments, one from $(0,0)$ to $(c,d)$ and the other one from $(c,d)$ to $(1,0)$), see Figure \ref{f:1}. 
\begin{figure}[h!]
\includegraphics[width=15.0em]{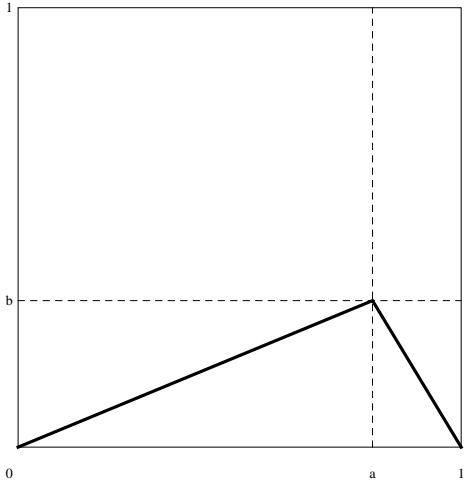}
\caption{The graph of a skew-tent function.}
\label{f:1}
\end{figure}

Let $A=\{0,a,b,f(b),f^2(b),f^3(b),\ldots \}$ and $B=\{0,c,d,g(d),g^2(d),\ldots \}$.
Both, $(f^n(b))$ and $(g^n(d))$ are strictly decreasing sequences with $\lim_{n\to \infty}f^n(b)=\lim_{n\to \infty}g^n(d)=0$.  
Therefore $A$ and $B$ are both closed in $[0,1]$. Clearly, $f$ is quasi Markov with respect to $(A,A)$ and $g$ is quasi Markov with respect to $(B,B)$. It is easily seen that $\{[0,1],f\}_{n=1}^{\infty}$ and $\{[0,1],g\}_{n=1}^{\infty}$ follow the same pattern with respect to $(A)_{n=1}^{\infty}\in \prod_{n=1}^{\infty}2^{[0,1]}$ and $(B)_{n=1}^{\infty}\in \prod_{n=1}^{\infty}2^{[0,1]}$. Therefore, by Corollary \ref{cor}, the inverse limits $\varprojlim\{[0,1],f\}_{n=1}^{\infty}$ and $\varprojlim\{[0,1],g\}_{n=1}^{\infty}$ are homeomorphic.

\end{example}
\begin{example}\label{ex:2}
Let $f_{p,q}:[0,1]\rightarrow [0,1]$ be a function defined by 
$$
f_{p,q} (t) =
			\begin{cases}
				qt \text{;} & t \leq p, \\
				\frac{1}{1-p}(1-t)  \text{;} & t > p,
			\end{cases}
$$
where $0 < q  < p < 1$.

\begin{figure}[h!]
\includegraphics[width=15.0em]{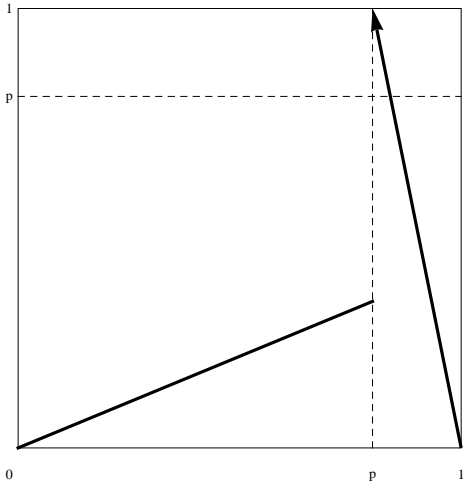}
\caption{The graph of $f_{p,q}$.}
\label{}
\end{figure}

Then, for all $p_1, q_1,p_2,q_2 \in (0,1)$, $ q_1  < p_1$ and $ q_2  < p_2$, the inverse limits $\varprojlim\{[0,1],f_{p_1,q_1}\}_{n=1}^{\infty}$ and $\varprojlim\{[0,1],f_{p_2,q_2}\}_{n=1}^{\infty}$ are homeomorphic since the following holds true. 
Obviously, 
$A=\{0,1,p_1,f_{p_1,q_1}(p_1),f_{p_1,q_1}^2(p_1),f_{p_1,q_1}^3(p_1),\ldots \}$ and $B=\{0,1,p_2,f_{p_2,q_2}(p_2),f_{p_2,q_2}^2(p_2),\ldots \}$
are non-empty closed subsets of $[0,1]$ that are both totally disconnected
(both, $(f_{p_1,q_1}^n(p_1))$ and $(f_{p_2,q_2}^n(p_2))$ are strictly decreasing sequences with $\lim_{n\to \infty}f_{p_1,q_1}^n(p_1)=\lim_{n\to \infty}f_{p_2,q_2}^n(p_2)=0$). 
It is easy to check that $f$ is quasi Markov with respect to $(A,A)$ and $g$ is quasi Markov with respect to $(B,B)$. Also, one can easily prove that $\{[0,1],f_{p_1,q_1}\}_{n=1}^{\infty}$ and $\{[0,1],f_{p_2,q_2}\}_{n=1}^{\infty}$ follow the same pattern with respect to $(A)_{n=1}^{\infty}\in \prod_{n=1}^{\infty}2^{[0,1]}$ and $(B)_{n=1}^{\infty}\in \prod_{n=1}^{\infty}2^{[0,1]}$. Therefore, by Corollary \ref{cor}, $\varprojlim\{[0,1],f_{p_1,q_1}\}_{n=1}^{\infty}$ and $\varprojlim\{[0,1],f_{p_2,q_2}\}_{n=1}^{\infty}$ are homeomorphic.
\end{example}

\begin{example}\label{ex:3}
Let $f:[0,1]\rightarrow [0,1]$ be the function defined by 
$$
f (t) =
			\begin{cases}
				\frac{1}{2}t \text{;} & t>0, \\
				1 \text{;} & t=0,
			\end{cases}
$$
for each $t\in [0,1]$.

Let $A=\{0,1,f(1),f^2(1),f^3(1),\ldots \}$.
It is easily seen that $f$ is a quasi Markov functions with respect to $(A,A)$ and that $\varprojlim\{[0,1],f\}_{n=1}^{\infty}$ is empty.
\end{example}
%

\bibliographystyle{amsplain}

Iztok Bani\v c and Matev\v z \v Crepnjak 

(iztok.banic@um.si and matevz.crepnjak@um.si)

Faculty of Natural Sciences and Mathematics, 

University of Maribor, 

Koro\v{s}ka 160, SI-2000 Maribor, Slovenia

\end{document}